\documentclass[12pt]{amsart}
\usepackage{graphicx}
\usepackage[headings]{fullpage}
\usepackage{amssymb,epic,eepic,epsfig,amsbsy,amsmath,amscd}
\numberwithin{equation}{section}
                        \theoremstyle{plain}
\usepackage{mathrsfs}

\newcommand\no[1]{}

\newtheorem{theorem}{Theorem}[section]

\newtheorem{lemma}[theorem]{Lemma}
\newtheorem{corollary}[theorem]{Corollary}
\newtheorem{proposition}[theorem]{Proposition}

\theoremstyle{definition}
\newtheorem{remark}[theorem]{Remark}

\def\BC{\mathbb C}

\def\BZ{\mathbb Z}

\def\BT{\mathbb T}

\def\CM{\mathcal M}

\def\CT{\mathcal T}

\def\la{\langle}
\def\ra{\rangle}

\DeclareMathOperator{\tr}{\mathrm tr}

\def\be { \begin{equation} }
\def\ee { \end{equation} }

\begin{document}

\title[Twisted Alexander polynomials with the adjoint action]
{Twisted Alexander polynomials with the adjoint action for some classes of knots}

\author[Anh T. Tran]{Anh T. Tran}
\address{Department of Mathematical Sciences, The University of Texas at Dallas, Richardson, TX 75080, USA}
\email{att140830@utdallas.edu}


\begin{abstract}
We calculate the twisted Alexander polynomial with the adjoint action for torus knots and  twist knots. As consequences of these calculations, we obtain the formula for the nonabelian Reidemeister torsion of torus knots in \cite{Du} and a formula for the nonabelian Reidemeister torsion of twist knots that is better than the one in \cite{DHY}.
\end{abstract}

\maketitle

\section{Introduction}

The Alexander polynomial, the first polynomial knot invariant, was discovered by Alexander in 1928 \cite{Al}. It was later interpreted in terms of Reidemeister torsions by Milnor \cite{Mi} and  Turaev \cite{Tu}. The twisted Alexander polynomial, a generalization of the Alexander polynomial, was introduced by Lin \cite{Li} for knots in $S^3$ and by Wada \cite{Wa} for finitely presented groups. It was also interpreted in terms of  Reidemeister torsions by Kitano \cite{Ki} and Kirk-Livingston \cite{KL}. As a consequence of this interpretation, one can calculate certain kinds of Reidemeister torsions of a knot from a finite presentation of its knot group by applying Fox differential calculus. 

In this paper we consider the twisted Alexander polynomial with the adjoint action. The adjoint action, $\text{Ad}$, is the conjugation on the Lie algebra $sl_2(\BC)$ by the Lie group $SL_2(\BC)$. Suppose $K$ is a knot and $G_K$ its knot group. For each representation $\rho:G_K \to SL_2(\BC)$, the composition $\text{Ad} \circ \rho: G_K \to SL_3(\BC)$ is a representation and hence, following \cite{Wa}, one can define a rational function $\Delta^{\text{Ad} \circ \rho}_{K}(t)$, called the twisted Alexander polynomial with the adjoint action associated to $\rho$. The twisted Alexander polynomial $\Delta^{\text{Ad} \circ \rho}_{K}(t)$ has been calculated for just a few knots \cite{DY}. The purpose of this paper is to calculate $\Delta^{\text{Ad} \circ \rho}_{K}(t)$ for  torus knots and twist knots, see Theorems \ref{Ad} and \ref{twist}. 

The paper is organized as follows. In Section 2 we review some backgrounds on the twisted Alexander polynomial with the adjoint action and state  the main results, Theorems \ref{Ad} and \ref{twist}, about the formulas for the twisted Alexander polynomial with the adjoint action for  torus knots and twist knots. We give proofs of Theorems \ref{Ad} and \ref{twist} in Sections \ref{pf-torus} and \ref{pf-twist} respectively. 

\section{The twisted Alexander polynomial with the adjoint action}

\subsection{Twisted Alexander polynomials} Let $K$ be a knot and $G_K=\pi_1(S^3\backslash K)$ its knot group. We fix a presentation
$$
G_K=
\langle a_1,\ldots,a_\ell~|~r_1,\ldots,r_{\ell-1}\rangle.
$$
(This might not be a Wirtinger representation, but must be of deficiency one.) 

Let $f:G_K\to H_1(S^3\backslash K,\BZ)
\cong {\BZ}
=\langle t
\rangle$ 
be the abelianization homomorphism and $\rho:G_K\to SL_k(\BC)$ a representation. These maps naturally induce two ring homomorphisms $\widetilde{f}:{\BZ}[G_K]\rightarrow {\BZ}[t^{\pm1}]$ and $\widetilde{\rho}: {\BZ}[G_K] \rightarrow \CM(k,{\BC})$, 
where ${\BZ}[G_K]$ is the group ring of $G_K$ 
and 
$\CM(k,{\BC})$ is the matrix algebra of degree $k$ over ${\BC}$. 
Then 
$\widetilde{\rho}\otimes\widetilde{f}: {\BZ}[G_K]\to \CM\left(k,{\BC}[t^{\pm1}]\right)$ 
is a ring homomorphism. 
Let 
$F_\ell$ be the free group on 
generators $a_1,\ldots,a_\ell$ and 
$\Phi:{\BZ}[F_\ell]\to \CM\left(k,{\BC}[t^{\pm1}]\right)$
the composition of the surjective map 
${\BZ}[F_\ell]\to{\BZ}[G_K]$ 
induced by the presentation of $G_K$ 
and the map 
$\widetilde{\rho}\otimes\widetilde{f}:{\BZ}[G_K]\to \CM(k,{\BC}[t^{\pm1}])$. 

We consider the $(\ell-1)\times \ell$ matrix $M$ 
whose $(i,j)$-component is the $k\times k$ matrix 
$$
\Phi\left(\frac{\partial r_i}{\partial a_j}\right)
\in \CM\left(k,{\BC}[t^{\pm1}]\right),
$$
where 
$\frac{\partial}{\partial a}$ 
denotes the Fox derivative. 
For 
$1\leq j\leq \ell$, 
let $M_j$ be
the $(\ell-1)\times(\ell-1)$ matrix obtained from $M$ 
by removing the $j$th column. 
We regard $M_j$ as 
a $k(\ell-1)\times k(\ell-1)$ matrix with coefficients in 
${\BC}[t^{\pm1}]$. 
Then Wada's twisted Alexander polynomial 
of the knot $K$ associated to the representation $\rho:G_K\to SL_k({\BC})$ 
is defined to be the rational function 
$$
\Delta^{\rho}_{K}(t)
=\frac{\det M_j}{\det\Phi(1-a_j)}. 
$$
It is defined 
up to a factor $t^{km}~(m\in{\BZ})$, see \cite{Wa}. 

\subsection{The twisted Alexander polynomial with the adjoint action} The adjoint action, $\text{Ad}$, is the conjugation on the Lie algebra $sl_2(\BC)$ by the Lie group $SL_2(\BC)$. For $A \in SL_2(\BC)$ and $g \in sl_2(\BC)$ we have $\text{Ad}_A(g)=AgA^{-1}$. For each representation $\rho: G_K \to SL_2(\BC)$, the composition $\text{Ad} \circ \rho: G_K \to SL_3(\BC)$ is a representation and hence  
one can define the twisted Alexander polynomial $\Delta^{\text{Ad} \circ \rho}_{K}(t)$. We call $\Delta^{\text{Ad} \circ \rho}_{K}(t)$ the twisted Alexander polynomial with the adjoint action associated to $\rho$. 

In this paper we are interested in the twisted Alexander polynomial with the adjoint action associated to irreducible/non-abelian $SL_2(\BC)$-representations.

\begin{remark} It is known that $\Delta^{\text{Ad} \circ \rho}_{K}(t)$ coincides with the nonabelian Reidemeister torsion polynomial $\CT^{\rho}_{K}(t)$ \cite{Ki, KL}. As a consequence of this identification, one can calculate the nonabelian Reidemeister torsion $\BT^{\rho}_{K}$ for any longitude-regular $SL_2(\BC)$-representation $\rho$ from a finite presentation of the knot group of $K$, by applying Fox differential calculus and the following formula 
\begin{equation} \label{yamaguchi}
\BT^{\rho}_{K}=-\lim_{t \to 1} \frac{\CT^{\rho}_{K}(t)}{t-1}
\end{equation} 
in \cite{Ya}. We refer the reader to \cite{Po, Du, DHY, DY} for definitions of $\CT^{\rho}_{K}(t)$ and $\BT^{\rho}_{K}$.  
\end{remark}

\subsection{Torus knots}

Let $K$ be the $(p,q)$-torus knot.The standard presentation for the knot group of $K$ is $G_K=\la c,d \mid c^p=d^q\ra$. Choose a pair $(r,s)$ of natural numbers such that $ps-qr=1$. Then $\mu=c^{-r}d^s$ is a meridian of $K$. Note that the abelian homomorphism $f: G_K \to H_1(S^3 \setminus K;\BZ) \cong \BZ = \la t\ra$ sends $c$ and $d$ to $t^q$ and $t^p$ respectively.

A representation $\rho:G_K\to SL_2(\BC)$ is called irreducible if there is no proper invariant line in $\BC^2$ under the action of $\rho(G_K)$. Let $R^{\text{irr}}(G_K)$ be the set of irreducible $SL_2(\BC)$-representations of $G_K$ and $\hat{R}^{\text{irr}}$ the set of conjugacy classes of representations in $R^{\text{irr}}(G_K)$. According to \cite[Prop. 34]{Jo}, we have the following description of $\hat{R}^{\text{irr}}(G_K)$. See also \cite{Kl, Le} for similar results.

\begin{proposition} \label{Johnson}
$\hat{R}^{\emph{irr}}(G_K)$ consists of $(p-1)(q-1)/2$ components, which are determined by the following data, denoted by $\hat{R}^{\emph{irr}}_{k,l}(G_K)$:
\begin{enumerate}
\item $0<k<p$, $0<l<q$, and $k \equiv l \pmod{2}$.

\item For every $[\rho] \in \hat{R}^{\emph{irr}}_{k,l}(G_K)$, we have $\rho(c^p)=\rho(d^q)=(-1)^kI$. Moreover, $\tr \rho(c)=2\cos \big( \frac{\pi k}{p} \big)$, $\tr \rho(d)=2\cos \big( \frac{\pi l}{q} \big)$, and $\tr \rho(\mu) \not= 2\cos \pi \big( \frac{rk}{p} \pm \frac{sl}{q} \big)$.
\end{enumerate}
In particular, $\hat{R}^{\emph{irr}}_{k,l}(G_K)$ is parametrized by $\tr \rho(\mu)$ and has complex dimension one.
\end{proposition}

Then we have the following.

\begin{theorem} \label{Ad}
Let $K$ be the $(p,q)$-torus knot. Suppose $\rho: G_K \to SL_2(\BC)$ is a representation such that $[\rho] \in \hat{R}^{\emph{irr}}_{k,l}$. Then 
$$\Delta^{\emph{Ad} \circ \rho}_{K}(t)=\frac{(t^{pq}-1)^3}{(t^p-1)(t^q-1)(t^{2q}-2(\cos\frac{2\pi k}{p})t^q+1)(t^{2p}-2(\cos\frac{2\pi l}{q})t^p+1)}.$$
\end{theorem}

It is known that for a torus knot every irreducible $SL_2(\BC)$-representation is longitude-regular, and hence one can define the non-abelian Reidemeister torsion $\BT^{\rho}_{K}$, see \cite{Po, Du}. Theorem \ref{Ad} and equation \eqref{yamaguchi} imply the following.

\begin{corollary} [\cite{Du}] \label{R-torsion}
Let $K$ be the $(p,q)$-torus knot. Suppose $\rho: G_K \to SL_2(\BC)$ is a representation such that $[\rho] \in \hat{R}^{\emph{irr}}_{k,l}$. Then
$$\BT^{\rho}_{K}=-\frac{p^2q^2}{16\sin ^2\big( \frac{\pi k}{p}\big)\sin ^2\big( \frac{\pi l}{q}\big)}.$$
\end{corollary}

\subsection{Twist knots}

Let $J(k,l)$ be the link in Figure 1, where $k,l$ denote 
the numbers of half twists in the boxes. Positive (resp. negative) numbers correspond 
to right-handed (resp. left-handed) twists. 
Note that $J(k,l)$ is a knot if and only if $kl$ is even, and is the trivial knot if $kl=0$. 
Furthermore, $J(k,l)\cong J(l,k)$ and 
$J(-k,-l)$ is the mirror image of $J(k,l)$. Hence we only consider $J(k,2n)$ for $k>0$ and $|n|>0$. When $k=2$, $J(2,2n)$ is the twist knot. For more information about $J(k,l)$, see \cite{HS}.

\begin{figure}[th]
\centerline{\psfig{file=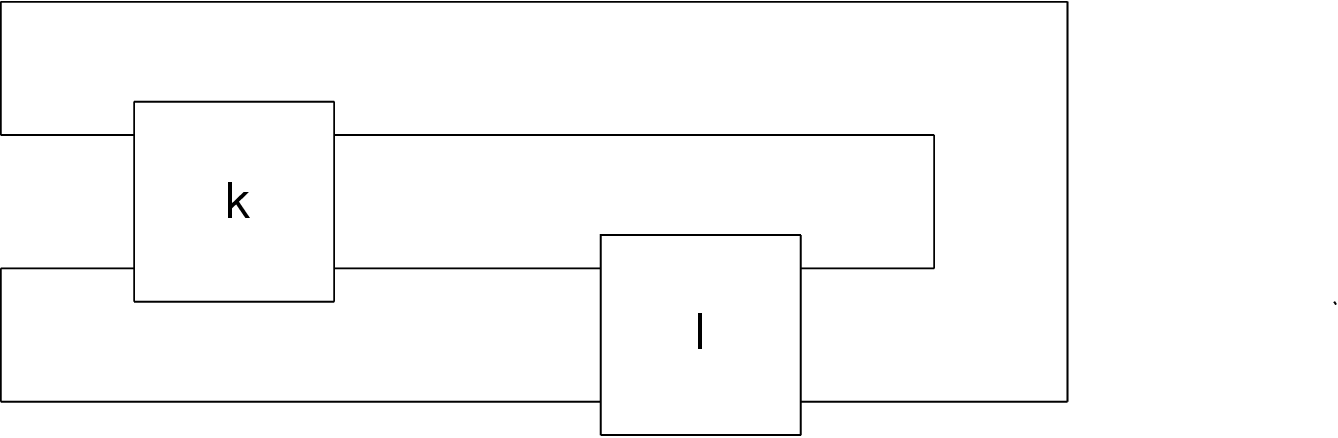,width=4in}}
\vspace*{8pt}
\caption{The link $J(k,l)$. }
\end{figure} 

Consider $K=J(2,2n)$. The knot group of $K$ is $G_K= \la a, b \mid w^na=bw^n \ra$ where $a,b$ are meridians and $w=ba^{-1}b^{-1}a$, see \cite{HS}. A representation $\rho:G_K\to SL_2(\BC)$ is called nonabelian if 
$\rho(G_K)$ is a nonabelian subgroup of $SL_2(\BC)$. 

Let $x=\tr \rho(a)=\tr \rho(b)$ and $y=\tr \rho(ab^{-1})$. Then we have the following.

\begin{theorem}  \label{twist}
Let $K$ be the twist knot $J(2,2n)$. Suppose $\rho: G_K \to SL_2(\BC)$ is a nonabelian representation. Then 
\begin{eqnarray*}
\Delta^{\text{Ad} \circ \rho}_{K}(t) &=& \frac{t-1}{(y+2-x^2)(y^2-yx^2+x^2)} \\
&&\times \left( nt^2 + \frac{(2n-1)y^2+yx^2-2nx^2(x^2-2)}{y^2-yx^2+2x^2} t+n \right).
\end{eqnarray*}
\end{theorem}

Theorem \ref{twist} and equation \eqref{yamaguchi} imply the following.

\begin{corollary} \label{coro}
Let $K$ be the twist knot $J(2,2n)$. Suppose $\rho: G_K \to SL_2(\BC)$ is a longitude-regular representation. Then 
$$\BT^{\rho}_{K}=\frac{-1}{(y+2-x^2)(y^2-yx^2+x^2)} \left(  \frac{(2n-1)y^2+yx^2-2nx^2(x^2-2)}{y^2-yx^2+2x^2} + 2n \right).$$
\end{corollary}

\begin{remark}
The nonabelian Reidemeister torsion for twist knots was calculated in \cite{DHY}. However, the formula in Corollary \ref{coro} is better.
\end{remark}

\section{Proof of Theorem \ref{Ad}} \label{pf-torus}

Recall that $K$ is the $(p,q)$-torus knot and $G_K=\la c,d \mid c^p=d^q \ra$ the standard presentation of its knot group. Suppose $\rho: G_K \to SL_2(\BC)$ is a representation such that $[\rho] \in \hat{R}^{\emph{irr}}_{k,l}$. Proposition \ref{Johnson} implies that the matrices $\rho(c)$ and $\rho(d)$ are respectively conjugate to $$\left[ \begin{array}{cc}
e^{i\frac{\pi k}{p}} & 0\\
0 & e^{-i\frac{\pi k}{p}} \end{array} \right] \quad \text{and} \quad\left[ \begin{array}{cc}
e^{i\frac{\pi l}{q}} & 0\\
0 & e^{-i\frac{\pi l}{q}} \end{array} \right].$$ 
By conjugation if necessary, we may assume that $\rho(c)=\left[ \begin{array}{cc}
\alpha & 0\\
0 & \alpha^{-1} \end{array} \right]$ and $\rho(d)$ is conjugate to $\left[ \begin{array}{cc}
\beta & 0\\
0 & \beta^{-1} \end{array} \right]$ where $\alpha=e^{i\frac{\pi k}{p}}$ and $\beta=e^{i\frac{\pi l}{q}}$. 

Let $\{E,H,F\}$ be the following usual $\BC$-basis of $sl_2(\BC)$:
$$E=\left[ \begin{array}{cc}
0 & 1\\
0 & 0 \end{array} \right],~H=\left[ \begin{array}{cc}
1 & 0\\
0 & -1 \end{array} \right],~F=\left[ \begin{array}{cc}
0 & 0\\
1 & 0 \end{array} \right].$$
Then the adjoint actions of $c$ and $d$ in the basis $\{E,H,F\}$ of $sl_2(\BC)$ are respectively given by the matrices $C=\text{Ad}_{\rho(c)}$ and $D=\text{Ad}_{\rho(d)}$, where $C=\text{diag}(\alpha^2, 1, \alpha^{-2})$ and $D$ is conjugate to $\text{diag}(\beta^2, 1, \beta^{-2})$.

 We have $\frac{\partial}{\partial c}c^pd^{-q}=1+c+ \cdots +c^{p-1}$, and hence
\begin{eqnarray*}
\Delta^{\text{Ad} \circ \rho}_{K}(t) &=& \frac{\det \Phi (\frac{\partial}{\partial c}c^pd^{-q})}{\det \Phi(d-1)}\\
&=& \frac{\det (I+t^q C +t^{2q} C^2+ \cdots + t^{(p-1)q}C^{p-1})}{\det (t^p D -I)}\\
&=& \frac{(1+\alpha^2 t^q + \cdots + \alpha^{2(p-1)} t^{(p-1)q})(1+\alpha^{-2} t^q + \cdots + \alpha^{-2(p-1)} t^{(p-1)q})}{t^{2p}-(\beta^2+\beta^{-2})t^p+1}\\
&& \times \frac{1+t^q + \cdots + t^{(p-1)q}}{t^p-1}.
\end{eqnarray*} 
Since $\alpha^{2p}=1$, we have $$1+\alpha^2 t^q + \cdots + \alpha^{2(p-1)} t^{(p-1)q}=\frac{(\alpha^2 t^q)^p-1}{\alpha^2 t^q-1}=\frac{t^{pq}-1}{\alpha^2 t^q-1}.$$
Similarly,
$$1+\alpha^{-2} t^q + \cdots + \alpha^{-2(p-1)} t^{(p-1)q}=\frac {t^{pq}-1}{\alpha^{-2} t^q-1}.$$
Hence 
\begin{eqnarray*}
\Delta^{\text{Ad} \circ \rho}_{K}(t) &=& \frac{(t^{pq}-1)^3}{(t^p-1)(t^q-1)(t^{2q}-(\alpha^2+\alpha^{-2})t^q+1)(t^{2p}-(\beta^2+\beta^{-2})t^p+1)}.
\end{eqnarray*} 
Theorem \ref{Ad} follows, since $\alpha^2+\alpha^{-2}=2\cos(\frac{2\pi k}{p})$ and $\beta^2+\beta^{-2}=2\cos(\frac{2\pi l}{q})$.

\begin{remark}
The above proof is similar to that of \cite[Thm. 1.1]{KM}. 
\end{remark}

\section{Proof of Theorem \ref{twist}} \label{pf-twist}

Recall that $K$ is the twist knot $J(2,2n)$ and $G_K= \la a, b \mid w^na=bw^n \ra$ its knot group, where $a,b$ are meridians and $w=ba^{-1}b^{-1}a$. 

\subsection{Nonabelian representations} 

Suppose $\rho: G_K \to SL_2(\BC)$ is a nonabelian representation. Taking conjugation if necessary, we can assume that $\rho$ 
has the form
$$
\rho(a)=\left[ \begin{array}{cc}
\sqrt{s} & \frac{1}{\sqrt{s}}\\
0 & \frac{1}{\sqrt{s}} \end{array} \right] \quad \text{and} \quad \rho(b)=\left[ \begin{array}{cc}
\sqrt{s} & 0\\
-\sqrt{s} \, u & \frac{1}{\sqrt{s}} \end{array} \right]
$$
where $(s,u) \in \BC^* \times \BC$ satisfies the Riley equation $\phi_K(s,u)=0$, see \cite{Ri}. Note that $x=\tr \rho(a)=\sqrt{s}+\frac{1}{\sqrt{s}}$ and $y=\tr \rho(ab^{-1})=u+2$.

Let $\gamma=\tr\rho(w)=2 + 2 u - \frac{u}{s} - s u + u^2$. By \cite{DHY} we have $$\phi_K(s,u)=(s+s^{-1}-1-u)\frac{\xi_+^n - \xi_-^n}{\xi_+ - \xi_-}-\frac{\xi_+^{n-1} - \xi_-^{n-1}}{\xi_+ - \xi_-}$$ where $\xi_{\pm}$ are eigenvalues of $\rho(w)$, i.e. $\xi_+\xi_-=1$ and $\xi_+ + \xi_- =\gamma$. 

Let
$$X=\frac{\xi_+^n - \xi_-^n}{\xi_+ - \xi_-},\quad Y=\frac{\xi_+^{n-1} - \xi_-^{n-1}}{\xi_+ - \xi_-}.$$
Since $\phi_K(s,u)=0$, we have $Y=(s+s^{-1}-1-u)X$. Moreover, it is easy to see that $X^2-\gamma XY+Y^2=1$. Hence we have the following.

\begin{lemma} \label{X}
 $$X^2=\frac{1}{1-(s+s^{-1}-1-u)\gamma+(s+s^{-1}-1-u)^2}.$$
\end{lemma}

\subsection{Adjoint action matrices}

Recall that $\{E,H,F\}$ is the following usual $\BC$-basis of $sl_2(\BC)$:
$$E=\left[ \begin{array}{cc}
0 & 1\\
0 & 0 \end{array} \right],~H=\left[ \begin{array}{cc}
1 & 0\\
0 & -1 \end{array} \right],~F=\left[ \begin{array}{cc}
0 & 0\\
1 & 0 \end{array} \right].$$
The adjoint actions of $a$ and $b$ in the basis $\{E,H,F\}$ of $sl_2(\BC)$ are given by the following matrices:
$$A=\text{Ad}_{\rho(a)}=\left[ \begin{array}{ccc}
s &-2 & -s^{-1}\\
0 & 1 & s^{-1}\\
0 & 0 & s^{-1}\end{array} \right], 
\quad B=\text{Ad}_{\rho(b)}=\left[ \begin{array}{ccc}
s & 0 & 0\\
su & 1 & 0\\
-su^2 & -2u & s^{-1}\end{array} \right].$$

Let $W=\text{Ad}_{\rho(w)}$. Note that the $SL_2(\BC)$-matrix $\rho(w)$ can be diagonalized by
$$Q=\left[ \begin{array}{cc}
u+1-s^{-1} & u+1-s^{-1}\\
1-su-\xi_+ & 1-su-\xi_- \end{array} \right].$$
Explicitly, $Q^{-1}\rho(w)Q$ is the diagonal matrix $\text{diag}(\xi_+, \xi_-)$. 

Let $\alpha=1-su-\xi_+,~\beta=1-su-\xi_-$ and $\delta=u+1-s^{-1}$. With respect to the basis $\{E,H,F\}$ of $sl_2(\BC)$, the matrix of the adjoint action of $Q$ is
$$P=\text{Ad}_Q=\frac{1}{\alpha-\beta}\left[ \begin{array}{ccc}
-\delta & 2\delta & \delta\\
\alpha & -(\alpha+\beta) & -\beta\\
\alpha^2/\delta & -2\alpha\beta/\delta & -\beta^2/\delta\end{array} \right].$$
Then $P^{-1}WP$ is the diagonal matrix $\text{diag}(\xi_+^2, 1, \xi_-^2)$.

Let $\Omega=I+W^{-1}+ \cdots + W^{-(n-1)}$. We have the following.

\begin{proposition} \label{O}
$$\Omega= \frac{1}{s^2u (1 - 2 s + s^2 - s u) (-4 s + u - 2 s u + s^2 u - s u^2)} 
\left[ 
\begin{array}{ccc}
\omega_{11} & \omega_{12} & \omega_{13}\\
\omega_{21} & \omega_{22} & \omega_{23}\\
\omega_{31} & \omega_{32} & \omega_{33} 
\end{array}
\right]$$
where
\begin{eqnarray*}
\omega_{11} &=& s^2 u \big\{ (1 - 2 s + s^2 - s u) (2 - 4 s + 2 s^2 + u - 6 s u + s^2 u - 4 s^3 u \\
&& - \, s u^2 + 3 s^2 u^2 - s^3 u^2 + s^4 u^2 - s^3 u^3) X^2 -2 n s (-1 + s + s u)^2\big\},\\
\omega_{12} &=& -2 s u (-1 + s + s u) \big\{ (1 - 2 s + s^2 - s u) \\
   && \times (-1 - 3 s^2 + 2 s u - s^2 u + s^3 u - 
   s^2 u^2) X^2 - n s (-1 + 2 s + s^2 + s u) \big\},\\
\omega_{13} &=& -(-1 + s + s u)^2 \big\{ (1 - 2 s + s^2 - s u) \\
            && \times (-2 s + u - s u + s^2 u - s u^2) X^2-2 n s^2 \big\},\\
\omega_{21} &=& s^2 u^2 (-1 + s + s u) \big\{ (1 - 2 s + s^2 - s u)  \\
   && \times (-1 - 3 s^2 + 2 s u - s^2 u + s^3 u - 
   s^2 u^2) X^2 -  n s (-1 + 2 s + s^2 + s u) \big\},\\
\omega_{22} &=& -s u \big\{ 2 (1 - 2 s + s^2 - s u) (-1 + s + s u)^2 (-2 s + u - s u + s^2 u - 
   s u^2) X^2 \\
   && - \, n s u (-1 + 2 s + s^2 + s u)^2 \big\},\\
\omega_{23} &=& -u (-1 + s + s u) \big\{  (1 - 2 s + s^2 - s u) (-1 + s + s u) \\
&&(-3 s + s^2 + u - 2 s u + s^2 u -
    s u^2) X^2 -n s^2 (-1 + 2 s + s^2 + s u)  \big\},\\
\omega_{31} &=& -s^2 u^2 (-1 + s + s u)^2 \big\{ (1 - 2 s + s^2 - s u) \\
     && \times (-2 s + u - s u + s^2 u - s u^2) X^2 -2 n s^2  \big\},\\
\omega_{32} &=& 2 s u^2 (-1 + s + s u)  \big\{ (1 - 2 s + s^2 - s u) (-1 + s + s u) \\
&&(-3 s + s^2 + u - 2 s u + s^2 u -
    s u^2) X^2 -n s^2 (-1 + 2 s + s^2 + s u) \big\},\\
\omega_{33} &=& u (-1 + s + s u) \big\{ (1 - 2 s + s^2 - s u) (-2 s^2 + 2 s^3 + 4 s u - 9 s^2 u + 3 s^3 u \\
&&- \, 
   u^2 + 4 s u^2 - 9 s^2 u^2 + 3 s^3 u^2 + 2 s u^3 - 4 s^2 u^3 + 
   s^3 u^3 - s^2 u^4) X^2\\
   && - \, 2 n s^3 (-1 + s + s u) \big\}.
\end{eqnarray*}
\end{proposition}

\begin{proof} Let 
\begin{eqnarray*}
d_1 &=& \xi_-^{n-1}+\xi_+^{n-1} = 2X-\gamma Y,\\ 
d_2 &=& \alpha \, \xi_-^{n-1}+\beta \, \xi_+^{n-1}=(1-su)(2X-\gamma Y)-\gamma X+(\gamma^2-2)Y,\\
d_3 &=& \alpha \, \xi_+^{n-1}+\beta \, \xi_-^{n-1}= (1-su)(2X-\gamma Y)-\gamma X+2Y,\\
d_4 &=& \alpha^2\xi_-^{n-1}+\beta^2\xi_+^{n-1}=(1-su)^2(2X-\gamma Y)-2(1-su)(\gamma X-(\gamma^2-2)Y)\\
    && \qquad \qquad \qquad \qquad \qquad + \, (\gamma^2-2)X-\gamma(\gamma^2-3)Y,\\
d_5 &=& \alpha^2\xi_+^{n-1}+\beta^2\xi_-^{n-1}=(1-su)^2(2X-\gamma Y)-2(1-su)(\gamma X-2Y)\\
    && \qquad \qquad \qquad \qquad \qquad + \, (\gamma^2-2)X-\gamma Y. 
\end{eqnarray*}
Since $P^{-1}WP$ is the diagonal matrix $\text{diag}(\xi_+^2, 1, \xi_-^2)$, we have
\begin{eqnarray*}
\Omega &=& P \, \text{diag}(\xi_-^{n-1}X,n,\xi_+^{n-1}X) \, P^{-1} \\
&=& \frac{1}{(\alpha-\beta)^2} \left[ \begin{array}{ccc}
-2\alpha\beta n+d_5X & 2\delta(-(\alpha+\beta)n+d_3X) & \delta^2(2n-d_1X)\\
\frac{\alpha\beta}{\delta}((\alpha+\beta)n-d_3X) & (\alpha+\beta)^2n-2\alpha\beta d_1X & \delta(-(\alpha+\beta)n+d_2X) \\
(\frac{\alpha\beta}{\delta})^2(2n-d_1X) & \frac{2\alpha\beta}{\delta}((\alpha+\beta)n-d_2X) & -2\alpha\beta n+d_4X \end{array} \right] .
\end{eqnarray*} 
Proposition \ref{O} follows from the above equation and $Y=(s+s^{-1}-1-u)X$.
\end{proof}

\subsection{Proof of Theorem \ref{twist}}

We focus on the case $n>0$. The case $n<0$ is similar. We have $$\frac{\partial}{\partial a} w^naw^{-n}b^{-1}=w^n(1+(1-a)(1+w^{-1}+ \cdots + w^{-(n-1)})(a^{-1}-a^{-1}b))$$ and hence
$$
\Delta^{\text{Ad} \circ \rho}_{K}(t) = \frac{\det \Phi(\frac{\partial}{\partial a} w^naw^{-n}b^{-1})}{\det \Phi(b-1)} 
= \frac{\det(I+(I-tA)\Omega(t^{-1}A^{-1}-A^{-1}B))}{\det (tB-I)}.
$$
Then, by  Lemma \ref{X} and Proposition \ref{O}, we have
\begin{eqnarray*}
\Delta^{\text{Ad} \circ \rho}_{K}(t) &=& \frac{s (t-1)}{(-1 + s - u) (1 - 2 s + s^2 - s u) (-1 + s + s u) (-4 s + u - 2 s u + 
   s^2 u - s u^2) t^3} \\
&& \times  \big\{ n s (-4 s + u - 2 s u + s^2 u - s u^2)t^2+ (2 n - 2 s + 4 n s - 4 n s^2 \\
&& - \, 2 s^3 + 4 n s^3 + 2 n s^4 - s u + 
  2 s^2 u - 8 n s^2 u - s^3 u + s^2 u^2 - 2 n s^2 u^2)t\\
  && + \, n s (-4 s + u - 2 s u + s^2 u - s u^2)\big\}\\
&=&  \frac{t-1}{(-2 + \frac{1}{s} + s - u)(-2 + \frac{1}{s} + s - 2 u + \frac{u}{s} + s u - u^2)  t^3}\\
&& \times \big\{ nt^2+ \frac{1}{-4 - 2 u + \frac{u}{s} + s u - u^2} \big(-4 n + \frac{2n}{s^2} - \frac{2}{s} + \frac{4n}{s} - 2 s + 4 n s \\
&&+\, 2 n s^2 + 2 u - 
 8 n u - \frac{u}{s} - s u + u^2 - 2 n u^2 \big)t+n \big\}.
\end{eqnarray*}
By substituting $s+\frac{1}{s}=x^2-2$, $s^2+\frac{1}{s^2}=x^4-4x^2+2$ and $u=y-2$ into the above equation, we get Theorem \ref{twist}.

\end{document}